\documentclass{monthlyarticle}

\usepackage{mathrsfs}
\usepackage{psfrag}
\usepackage{mathrsfs}
\usepackage{amsmath}
\usepackage{url}
\usepackage{amssymb,amsthm}
\usepackage{graphicx,color}
\usepackage[margin=4cm]{geometry}
\usepackage{epstopdf}
\usepackage[backgroundcolor=green!40, linecolor=green!40]{todonotes}
\usepackage{subcaption}
\usepackage{xypic}

\def\({\left(}
\def\){\right)}

\def\leq{\leqslant}
\def\geq{\geqslant}

\newcommand{\T}{\widetilde{T}}
\newcommand{\N}{\mathbb{N}}

\newcommand{\tr}{\widehat{T}}
\newcommand{\xr}{\widehat{X}}

\newcommand{\og}{\mathfrak{O}_G}
\newcommand{\ot}{\mathfrak{O}_T}

\newcommand{\CreatingClosedOrbits}{
\begin{tikzpicture}
\filldraw (0,0) circle[radius=0.05];
\filldraw (.5,-.5) circle[radius=0.05];
\filldraw (1,0) circle[radius=0.05];
\draw[->,blue] (0,0) [out=270,in=180] to (0.4,-0.5);
\draw[->,blue] (0.5,-0.5) [out=350,in=260] to (1,-0.1);
\draw[blue] (1,0) [out=90,in=360] to (0.5,0.5);
\draw[->,blue] (0.5,0.5) [out=180,in=90] to (0,0.1);
\filldraw (0,1.5) circle[radius=0.05];
\filldraw (0,2.5) circle[radius=0.05];
\filldraw (1,2) circle[radius=0.05];
\filldraw (0.5,1.6) circle[radius=0.05];
\draw[->, blue] (0,1.5) [out=0,in=220] to (0.425,1.6);
\draw[->, blue] (0.5,1.6) [out=30,in=230] to (0.95,1.95);
\draw[->, blue] (1,2) [out=120,in=360] to (0.1,2.5);
\draw[->] (0,2.5) -- (0,0.1);
\draw[->] (0.5,1.6) -- (0.5,-0.4);
\draw[->] (1,2) -- (1,0.1);
\node[right, blue] at (1.25,2) {$ T $};
\node[right, blue] at (1.25,0.25) {$ T' $};
\end{tikzpicture}}

\newcommand{\SquashingClosedOrbits}{
\begin{tikzpicture}
\filldraw (0,0.5) circle[radius=0.05];
\filldraw (0,1.5) circle[radius=0.05];
\filldraw (0,2.1) circle[radius=0.05];
\filldraw (0,2.7) circle[radius=0.05];
\filldraw (0,3.3) circle[radius=0.05];
\draw[->,blue] (0,1.5) [out=45,in=315] to (0.05,2.15);
\draw[->,blue] (0,2.1) [out=45,in=315] to (0.05,2.65);
\draw[->,blue] (0,2.7) [out=45,in=315] to (0.05,3.25);
\draw[->,blue] (0,3.3) [out=210,in=150] to (-0.1,1.5);
\draw[->] (0,1.5) -- (0,0.6);
\draw[blue] (0,0.5) [out=180,in=90] to (-0.25,0.25);
\draw[blue] (-0.25,0.25) [out=270,in=180] to (0,0);
\draw[blue] (0,0) [out=360,in=270] to (0.25,0.25);
\draw[->,blue] (0.25,0.25) [out=90,in=360] to (0.075,0.5);
\node[right,blue] at (0.5,2.4) {$ T $};
\node[right,blue] at (0.5,0.5) {$ T' $};
\end{tikzpicture}
}

\newtheorem{theorem}{Theorem}
\newtheorem{lemma}[theorem]{Lemma}
\newtheorem{corollary}[theorem]{Corollary}
\newtheorem{proposition}[theorem]{Proposition}

\theoremstyle{definition}

\newtheorem{example}[theorem]{Example}

\newtheorem{remark}[theorem]{Remark}

\begin{document}

\title{Closed orbits in quotient systems}
\author{Stefanie Zegowitz}

\maketitle

\begin{abstract}
\noindent We study the relationship between pairs of topological dynamical systems $ (X,T) $ and $ (X',T') $ where $ (X',T') $ is the quotient of $ (X,T) $ under the action of a finite group $ G $. We describe three phenomena concerning the behaviour of closed orbits in the quotient system, and the constraints given by these phenomena. We find upper and lower bounds for the extremal behaviour of closed orbits in the quotient system in terms of properties of $ G $ and show that any growth rate in between these bounds can be achieved. 
\end{abstract}

\section{Introduction}
For a continuous map $ T:X\to X $, a closed orbit of length $ n $ is any set of the form 
\begin{equation*} 
\mathfrak{O}_T(x)= \{ x, T(x), T^2(x), T^3(x), ... , T^{n}(x)=x \}\,,
\end{equation*}
with cardinality $ |\ot(x)|=n $, for $ n $ a natural number.

The study of closed orbits is useful for understanding the growth properties of the dynamical system $ (X,T) $, and useful analogies have been drawn between the study of closed orbits of $ (X,T) $ and the study of prime numbers. Dynamical analogues of the Prime Number Theorem concern the asymptotic behaviour of quantities like
\begin{equation}
\label{asymptote}
\pi_T(n)= \# \{ \ot(x): |\mathfrak{O}_T(x)|\leq n  \}\,.
\end{equation}
Waddington proves in \cite{Waddington} that, for $ T $ an ergodic automorphism of a torus, we have that
\begin{equation*}
\pi_T(n)\sim \frac{e^{h(T)(n+1)}}{n} \rho(n) \quad\text{as } n\to\infty\,,
\end{equation*}
where $ \rho:\N\to\mathbb{R}^+ $ is an explicit almost periodic function which is bounded away from zero and infinity ($ \rho $ is constant for an expansive automorphism), where $ h(T) $ denotes the topological entropy of $ T $. Dynamical analogues of Mertens' Theorem concern the asymptotic behaviour of quantities like
\begin{equation}
\label{Mertens}
M_T(n)= \sum\limits_{|\mathfrak{O}_T(x)|\leq n} \frac{1}{e^{h(T)|\ot(x)|}}.
\end{equation}
In \cite{Mertens} and \cite{Noorani}, it is shown that, for $ T $ an ergodic quasi-hyperbolic toral automorphism, we have
\begin{equation*}
M_T(n)\sim C\log(n)\quad\text{as }n\to\infty\,.
\end{equation*}

% flows

General statements regarding the growth properties of both \eqref{asymptote} and \eqref{Mertens} in flows are proved by Parry in \cite{Parry}, by Parry and Pollicott in \cite{ParryPollicott}, and by Sharp in \cite{SharpMertens}. In \cite{Sharp}, Sharp obtains results for the asymptotic behaviour of (\ref{asymptote}) in the context of finite group extensions of Axiom A flows. Assuming that $ T:\mathbb{R}\times M\to M $ is an Axiom A flow, where $ M $ is a compact smooth Riemannian manifold, Sharp considers a finite group $ G $ acting freely on $ M $ and commuting with $ T $. Here, $ T $ induces an Axiom A flow $ T' $ on the quotient space $ M'= G\backslash M $, and we obtain a topological semi-conjugacy between $ T $ and $ T' $, that is we have a continuous surjection $ f:M\to M' $ with $ T\circ f=f\circ T' $. Each closed orbit $ \mathfrak{O}_{T'}(x) $ gives rise to a unique conjugacy class in $ G $, and we denote this conjugacy class by $ \mathcal{C}_{\mathfrak{O}_{T'}(x)} $. A special case of the expression which Sharp studies is the quantity
\begin{flalign}
\label{Sharp}
 \pi_{\mathcal{C}}(n)= \#\{ \mathfrak{O}_{T'}(x): |\mathfrak{O}_{T'}(x)|\leq n , \mathcal{C}_{\mathfrak{O}_{T'}(x)}= \mathcal{C} \}\,,
\end{flalign} 
when $ T' $ is topologically weak-mixing, and he obtains the result that
\begin{equation*}
\pi_{\mathcal{C}}(n) \sim \frac{|\mathcal{C}|}{|G|}\,\pi_{T'}(n)\quad \text{as $ n\to\infty $}\,.
\end{equation*}

% My work

The above results for flows were motivational in studying a discrete analogue. We take $ (X,d) $ to be a compact metric space and $ T:X\to X $ to be a homeomorphism, that is $ (X,T) $ is a \textbf{topological dynamical system}. We consider a finite group $ G $ acting on $ X $, where the action of $ G $ commutes with $ T $. We do not require the action to be free. We denote by $ X'=G\backslash X $ the quotient space and by $ T':X'\to X' $ the induced map on the quotient space. We call the system $ (X',T') $ the quotient system of the topological dynamical system $ (X,T) $, and we have defined a topological semi-conjugacy between $ T $ and $ T' $ given by the quotient map itself. Note that this construction is an extension of the $ C_2 $ case analysed in \cite{Halving} by Stevens, Ward, and Zegowitz. 

Write
\[
\mathcal{F}_n(T)= \{x\in X: T^n(x)=x\} 
\]
for the set of points of period $ n $ under $ T $ and $ F_n(T)=\# \mathcal{F}_n(T) $ for the number of points of period $ n $. Similarly, write
\[
\mathcal{O}_n(T)= \{\ot(x): |\ot(x)|=n\}
\]
for the set of closed orbits of length $ n $ under $ T $ and $ O_n(T) =\# \mathcal{O}_n(T) $ for the number of closed orbits of length $ n $. Note that for two topologically semi-conjugate maps $ T $ and $ T' $ there is, in general, no relationship between the count of periodic points (and closed orbits) of $ T $ and $ T' $. For example, it is possible for $ T $ to have no closed orbits while $ T' $ has many closed orbits as illustrated in Figure \ref{IntroFigure1}. On the other hand, it is possible for $ T $ to have many closed orbits while $ T' $ has few as illustrated in Figure \ref{IntroFigure2}.
\begin{figure}[ht]
\begin{minipage}[b]{0.45\linewidth}
\centering
\CreatingClosedOrbits
\caption{Creating closed orbits from non-closed orbits}
\label{IntroFigure1}
\end{minipage}
\begin{minipage}[b]{0.1\linewidth}
\centering
\qquad
\end{minipage}
\begin{minipage}[b]{0.43\linewidth}
\centering
\SquashingClosedOrbits
\caption{Squashing closed orbits to a single closed orbit of length $ 1 $ (a fixed point)}
\label{IntroFigure2}
\end{minipage}
\end{figure}
However, if we consider two topologically semi-conjugate maps $ T $ and $ T' $, where $ (X',T') $ is the quotient system of the dynamical system $ (X,T) $ under the action of a finite group $ G $, it is then possible to establish a relationship between the count of periodic points (and closed orbits) of $ T $ and $ T' $. 

\begin{example}
\label{ExampleCircleDoublingMap}
Let $ T:\mathbb{R}^2/\mathbb{Z}^2 \to \mathbb{R}^2/\mathbb{Z}^2 $ be the doubling map defined by
\begin{equation*}
T(x,y)= (2x,2y)\quad(\text{mod }1)\,,
\end{equation*}
for all $ (x,y)\in X= \mathbb{R}^2/\mathbb{Z}^2 $, and let $ D_8 $ be the dihedral group of eight elements acting on $ X $ as the symmetries of the square $ [0,1]^2 $, so that the action commutes with $ T $. Then \cite{Zegowitz} shows that
\[
F_n(T)=(2^n-1)^2 \quad\text{and}\quad F_n(\tr)=4^n , 
\]
for all $ n\geq 1 $, where $ \tr $ is the induced map on the quotient space. It follows that
\[
F_n(T)\sim F_n(\tr)\quad\text{as $ n\to\infty $},
\]
showing the same asymptotic growth rates for closed orbits of $ (X,T) $ and its quotient system.
\end{example}

We find that Example \ref{ExampleCircleDoublingMap} is not representative in a general setting where orbit behaviour can be much more complicated, but nevertheless we are able to analyse the possible behaviours by partitioning the space $ X $ according to the action of the group $ G $. In particular, there are two numbers  $ \nabla $ and $ \Theta $ associated to the group $ G $ which determine the extremal orbit behaviour as follows: 

\begin{theorem}
\label{MainTheorem}
Let $ (a_n) $ be a sequence of non-negative integers with $ a_1\geq 1 $ such that there exists $ N>0 $ with $ \frac{a_{\nabla n}}{\Theta} \geq 	a_n $ for $ n\geq N $. Further, let $ (b_n) $ be any sequence of non-negative integers such that $ b_1> \frac{a_1}{|G|} $ and
\[
\begin{cases}
\frac{a_n}{|G|}\leq b_n \leq a_n, & \text{for $ n<N $}, \\
a_n\leq b_n \leq \frac{a_{\nabla n}}{\Theta}, & \text{for $ n\geq N $}.
\end{cases}
\]
Then there exist a topological dynamical system $ (X,T) $ and an action of $ G $ on $ X $ which commutes with $ T $ such that
\[
O_n(T)=a_n \qquad\text{and}\qquad O_n(T')=b_n,
\]
for all $ n\geq 1 $.
\end{theorem}

Section 2 gives background information on quotient systems. Section 3 describes the three orbit behaviour phenomena which occur in quotient systems-- surviving orbits, glueing orbits, and shortening orbits-- and establishes upper and lower bounds for closed orbits in quotient systems. Section 4 introduces the Super Basic Lemma, and Section 5 analyses growth rates for closed orbits in quotient systems and contains the proof for Theorem \ref{MainTheorem}.

Please note that throughout this paper, when we refer to orbits of the dynamical system $ (X,T) $, we refer to closed orbits.

\section{Quotient Systems}

Suppose $ (X,d) $ is a compact metric space, $ (X,T) $ is a topological dynamical system with $ F_n(T)<\infty $ for all $ n\geq 1 $, and $ G $ is a finite group acting continuously on $ X $ where the action of $ G $ commutes with $ T $. 

Then the orbit of $ x\in X $ under the action of $ G $ is defined by 
\begin{equation*} 
\mathfrak{O}_G(x)= \{ g(x): g\in G\},
\end{equation*}
and the set of all orbits of $ X $ under the action of $ G $, that is the \textbf{quotient space}, is defined by
\begin{equation*}
X'=G\backslash X= \{ \mathfrak{O}_G(x): x\in X\}.
\end{equation*}
We define $ \pi: X\to X' $ to be the canonical map $ \pi(x)=\mathfrak{O}_G(x) $, for all $ x\in X $, and note that $ \pi $ is surjective and, since $ G $ is finite, all orbits are finite. In particular, $ \pi $ induces a surjective map
\begin{equation}
\label{orbitsurjectivity}
\bigsqcup_{n\geq 1} \mathcal{O}_n(T) \to \bigsqcup_{n\geq 1} \mathcal{O}_n(T').
\end{equation}
To guarantee the continuity of $ \pi $, we define a metric $ d' $ on $ X' $ by
\begin{equation*}
d'(\mathfrak{O}_G(x),\mathfrak{O}_G(y)) =\min\{ d(x,y): x\in \og(x), y\in \og(y) \},
\end{equation*}
for all $ x\in X $. It is easily verifiable that $ d' $ is a metric and that $ \pi $ is continuous with respect to $ d' $, but the reader may refer to \cite{Zegowitz} for a proof. 

Further, we define the induced map $ T':X'\to X' $ on the quotient space to be
\[
T'(\mathfrak{O}_G(x))=\mathfrak{O}_G(T(x)) .
\] 
Then $ \pi(T(x))=\mathfrak{O}_G(T(x))= T' (\mathfrak{O}_G(x))= T'(\pi(x)) $, and $ T $ and $ T' $ are topologically semi-conjugate. We call $ \pi $ a \textbf{topological factor map} and $ (X',T') $ the \textbf{quotient system} of $ (X,T) $. Note that $ T' $ is well defined since if $ \mathfrak{O}_G(x)=\mathfrak{O}_G(y) $ then $ T'(\mathfrak{O}_G(x))= T'(\mathfrak{O}_G(y)) $. This holds since $ g $ commutes with $ T $. Again, it is easily verifiable that $ T' $ is continuous with respect to $ d' $, but the reader may refer to \cite{Zegowitz} for a proof. Also note that since $ G $ is finite, we have that $ (X',T') $ is a topological dynamical system.

\begin{example}
\label{ExampleTriangleMap}
Let $ T $ be the doubling map defined on the two-torus $ X=\mathbb{R}^2/\mathbb{Z}^2 $, and let $ D_8 $ be the dihedral group of eight elements acting on $ X $ as the symmetries of the square $ [0,1]^2 $, so that the action commutes with $ T $. The induced map on the quotient space, which can be identified with $ \xr =\{ (x,y): 0\leq y\leq x\leq \frac{1}{2} \} $ as illustrated in Figure \ref{TriangleQuotientSpace}, is given by the triangle map $ \tr $ defined by
\[
\tr(x,y)=
\begin{cases}
(2x,2y)\,, & \text{if }  0\leq y\leq x\leq \tfrac{1}{2} , \\
(1-2x,2y)\,, & \text{if }  \tfrac{1}{4}\leq x\leq\tfrac{1}{2}, 0\leq y\leq \tfrac{1}{4}\text{ and } y\leq \tfrac{1}{2}-x  , \\
(2y,1-2x)\,, & \text{if } \tfrac{1}{4}\leq x\leq\tfrac{1}{2}, 0\leq y\leq \tfrac{1}{4}\text{ and } y\geq \tfrac{1}{2}-x , \\
(1-2y,1-2x)\,, & \text{if }  \tfrac{1}{4}\leq y\leq x\leq \tfrac{1}{2}.
\end{cases}
\]

\begin{figure}
\begin{center}
\begin{tikzpicture}
% square 1 
\draw (0,0) -- (0,2);
\draw (0,2) -- (2,2);
\draw (2,2) -- (2,0);
\draw (2,0) -- (0,0);
\draw (-0.1,0.975)-- (0.1,0.975);
\draw (-0.1,1.025)-- (0.1,1.025);
\draw (1.9,0.975)-- (2.1,0.975);
\draw (2.1,1.025)-- (1.9,1.025);
\draw (0.975,-0.1)-- (0.975,0.1);
\draw (1.025,-0.1)-- (1.025,0.1);
\draw (0.975,1.9)-- (0.975,2.1);
\draw (1.025,1.9)-- (1.025,2.1);
\node[below] at (0.1,0) {$ 0 $};
\node[below] at (2,0) {$ 1 $};
% middle
\draw[dashed, thick, blue] (-0.25,-0.25) -- (2.25,2.25);
% arrow
\draw[<->, blue] (-0.75,0.25) [out=270,in=180] to (0.25,-0.75);
\draw[->, thick] (2.25,0.5) -- (2.75,0.5);
% square 2
\draw (3,0) -- (5,2);
\draw (5,2) -- (5,0);
\draw (5,0) -- (3,0);
\draw (4.9,0.975)-- (5.1,0.975);
\draw (4.9,1.025)-- (5.1,1.025);
\draw (3.975,-0.1)-- (3.975,0.1);
\draw (4.025,-0.1)-- (4.025,0.1);
\node[below] at (3,0) {$ 0 $};
\node[below] at (4.9,0) {$ 1 $};
% middle 2
\draw[dashed, thick, blue] (3.75,1.25) -- (5.25,-0.25);
% arrow 2
\draw[<->, blue] (3.25,0.75) [out=90,in=180] to (4.25,1.75);
\draw[->, thick] (5.25,0.5) -- (5.75,0.5);
% square 3
\draw (6,0) -- (7,1);
\draw (7,1) -- (8,0);
\draw (8,0) -- (6,0);
\node[below] at (6,0) {$ 0 $};
\node[below] at (8,0) {$ 1 $};
\node[left] at (7,0.35) {$ \frac{1}{2} $};
% middle 3
\draw[dashed, thick, blue] (7,-0.25) -- (7,1.3);
% arrow 3
\draw[<->, blue] (6.25,1.1) [out=60,in=120] to (7.75,1.1);
\draw[->, thick] (8.25,0.5) -- (8.75,0.5);
% square 4
\draw (9,0)-- (10,1);
\draw (10,1)-- (10,0);
\draw (10,0) -- (9,0);
\node[below] at (9,0) {$ 0 $};
\node[below] at (10,0) {$ \frac{1}{2} $};
\end{tikzpicture}
\end{center}
\caption{}
\label{TriangleQuotientSpace}
\end{figure}
\end{example}

\section{Surviving, Glueing and Shortening Orbits}

We have that three phenomena occur in quotient systems:
orbits which remain unchanged (surviving orbits), orbits which glue together with other orbits of the same length to form into one orbit (glueing orbits), and orbits which shorten in length (shortening orbits). Note that it is possible for the two phenomena, glueing and shortening, to occur at the same time. We have the following two lemmas concerning the three phenomena:

% shortening

\begin{lemma}
\label{shortening}
Let $ T:X\to X $ be a map defined on a set $ X $. Let $ G $ be a finite group acting on $ X $ where the action of $ G $ commutes with $ T $ and let $ k\in\N $. Then the orbit $ \mathfrak{O}_T(x)$  of $ x\in X $ shortens in length by a factor of $ \frac{1}{k} $ if and only if $ |\mathfrak{O}_T(x)\cap\mathfrak{O}_G(x)|=k $. Moreover, if $ \mathfrak{O}_T(x) $ shortens in length by a factor of $ \frac{1}{k} $, then there exists an element $ g\in G $ such that $ \mathfrak{O}_T(x)\cap \mathfrak{O}_G(x)= \{g^i(x): i\in\N_0 \}= \mathfrak{O}_{\left\langle g\right\rangle}(x) $.
\end{lemma}

\begin{proof}
($ \Rightarrow $) Suppose $ \mathfrak{O}_T(x)$ shortens in length by a factor of $ \frac{1}{k} $. Then $ |\mathfrak{O}_T(x)|= km $, for some $ m $, and $ |\mathfrak{O}_{T'}(x)|=m $. Then $ T^m(x)=g(x) $, for some $ g\in G $, where $ m $ is the least integer for which there is such a $ g $. It follows that 
\begin{equation*}
\{x, T^m(x), T^{2m}(x), \ldots ,T^{(k-1)m}(x) \}\subseteq \left[\mathfrak{O}_T(x)\cap\mathfrak{O}_G(x)\right]\,.
\end{equation*}
Now, suppose there is $ j $ such that $ T^j(x) \in \left[\mathfrak{O}_T(x)\cap\mathfrak{O}_G(x)\right] $. Then $ T^j(x)=g'(x) $, for some $ g'\in G $. If $ m \nmid j $, then $ j=lm+s $, where $ 0<s<m $. But then $ T^s(x)=g(g')^{-l}(x) $, and this contradicts the minimality of $ m $. Hence, $ m\mid j $, and $ g'(x)=g^i(x) $, for some $ i $. The result follows.

($ \Leftarrow $) Suppose $ |\og(x)\cap \ot(x)|=k $. There exists $ m $ such that $ T^m(x)=g(x) $, for some $ g\in G $. Assume $ m $ to be minimal with this property. Then 
\begin{equation*}
\{ x, T^m(x), T^{2m}(x), \ldots ,T^{(k-1)m}(x), \ldots \} \subseteq \left[\mathfrak{O}_T(x)\cap\mathfrak{O}_G(x)\right]\,. 
\end{equation*}
By a similar argument as for ($ \Rightarrow $), if there exists $ j>m $ such that $ T^j(x) \in \left[\mathfrak{O}_T(x)\cap\mathfrak{O}_G(x)\right] $, then $ m\mid j $ and
\begin{flalign*}
\left[\mathfrak{O}_T(x)\cap\mathfrak{O}_G(x)\right] &= \{ x, T^m(x), T^{2m}(x), \ldots ,T^{(k-1)m}(x), \ldots \} \\
&=\{ x, g(x), g^{2}(x), \ldots ,g^{(n-1)}(x), \ldots \}\,.
\end{flalign*}
Since $ |\mathfrak{O}_T(x)\cap\mathfrak{O}_G(x)|=k $ by assumption, it follows that $ g^k(x)=x $, where $ k $ is minimal. The result follows.
\end{proof}

\begin{remark}
Note that by Lemma \ref{shortening} and the Orbit--Stabilizer--Theorem, if the orbit $ \ot(x) $ shortens in length by a factor of $ \frac{1}{k} $, then there exists $ g\in G $ such that 
\[
k= |\mathfrak{O}_T(x)\cap \mathfrak{O}_G(x)|= |\mathfrak{O}_{\left\langle g\right\rangle}(x)|\leq |\left\langle g\right\rangle |.
\]
Then contrary to the natural assumption that the maximum factor by which the orbit $ \ot(x) $ can shorten in length is given by $ \frac{1}{|G|} $, we find that the maximum factor is determined by the maximal order of an element of $ G $. It follows that an orbit $ \ot(x) $ can shorten in length by a factor of $ \frac{1}{|G|} $ only if there exists $ g\in G $ such that $ |\left\langle g\right\rangle|=|G| $.
\end{remark}

% glueing

\begin{lemma}
\label{glueing}
Let $ T:X\to X $ be a map defined on a set $ X $. Let $ G $ be a finite group acting on $ X $ where the action of $ G $ commutes with $ T $, and let $ x\in X $. Then the number of orbits that glue to $ \mathfrak{O}_T(x) $ (including itself) is given by  
\begin{equation*}
\frac{|\mathfrak{O}_G(x)|}{|\mathfrak{O}_T(x)\cap\mathfrak{O}_G(x)|}\,.
\end{equation*}
\end{lemma}

\begin{proof}
Set $ |\og(x)|=m $. Then $ |\og(x)\cap \ot(x)|=k $, for some $ 1\leq k\leq m $. By Proposition~\ref{shortening}, it follows that there exists an element $ g\in G $ such that $ |\og(x)\cap\ot(x)|=|\mathfrak{O}_{\left\langle g\right\rangle}(x)|=k $. Then $ k $ divides $ |\og(x)| $ and $ m=kl $, for some $ l $. Hence, the number of glueing orbits is given by $ \frac{m}{k}= \frac{|\og(x)|}{|\og(x)\cap \ot(x)|}=l $.
\end{proof}

% surviving orbits

\begin{remark}
Note that if $ |\mathfrak{O}_G(x)| =1 $, then $ \ot(x) $ is a surviving orbit-- it neither shortens nor glues.
\end{remark}

To determine what orbit behaviours occur in a particular quotient system $ (X',T') $, we need to partition $ X $ according to the $ G $-conjugacy classes of subgroups of $ G $. We define 
\begin{equation*}
G_x = \{g\in G: g(x)=x \}.
\end{equation*}
to be the \textbf{stabilizer} of $ x $, and 
\begin{equation*}
N_G(H)= \{ g\in G: gHg^{-1} =H \}
\end{equation*}
to be the \textbf{normalizer} of a subgroup $ H\leq G $. Further, by $ P(G) $, we denote the set of all subgroups of $ G $, and by $ \overline{P}(G) $, we denote the set of all $ G $-conjugacy classes of subgroups of $ G $. Then for $ H\in P(G) $, we write 
\begin{equation*}
[H]= \{ gHg^{-1}: g\in G \}
\end{equation*}
for its conjugacy class in $ \overline{P}(G) $. Further, we let $ X_{H} $ be the set
\begin{equation*}
X_{H}= \{x\in X: G_x = H \}
\end{equation*}
and $ X_{[H]} $ be the set
\begin{equation*}
X_{[H]}= \{x\in X: G_x\in [H] \}.
\end{equation*}
Note that if $ g(x)=x $ then $ g(T(x))=T(x) $, for all $ x\in X $ and $ g\in G $, so that $ G_x \subseteq G_{T(x)} $. Moreover, if $ x $ is periodic, then if $ g(T(x))=T(x) $ then $ g(x)=x $. Hence, if we let $ x\in\mathbb{X} $, where $ \mathbb{X} $ is the \textbf{periodic set}, that is $ \mathbb{X} $ is the subset of all periodic points of $ X $ under $ T $, then $ G_x= G_{T(x)} $. Then $ T $ preserves each set $ \mathbb{X}_{[H]}=X_{[H]}\cap\mathbb{X} $, and we have
\begin{equation}
\label{PartitioningPeriodicSet}
\mathbb{X}= \bigsqcup_{[H]\in \overline{P}(G)} \mathbb{X}_{[H]}.
\end{equation}
For the induced map $ T':\mathbb{X}'\to \mathbb{X}' $, the quotient set is partitioned as follows:
 \begin{flalign*}
\mathbb{X}' = \bigsqcup_{[H]\in \overline{P}(G)} G\backslash \mathbb{X}_{[H]}= \bigsqcup_{[H]\in \overline{P}(G)} \mathbb{X}'_{[H]}\,.
\end{flalign*}
Again, $ T' $ preserves each set $ \mathbb{X}'_{[H]} $. Now, for any group $ H $, set
\[
\Delta(H)=\{|\left\langle h\right\rangle|:h\in H\}.
\]
We have the following result concerning orbit behaviour:

\begin{lemma}
\label{orbitbehaviour}
Given the orbit $ \ot(x) $ of $ x\in X_{[H]} $, there exists  $ k\in\Delta\left( N_G(H)/H\right) $ such that 
\begin{itemize}
\item[(i)] $ \ot(x) $ shortens in length by a factor of $ \frac{1}{k} $, and
\item[(ii)] $ \ot(x) $ glues to $ \frac{[G:H]}{k} $ orbits (including itself).
\end{itemize}
\end{lemma}

\begin{proof}
Let $ x\in X_H $, for some $ H\leq G $. (i) Suppose $ |\ot(x)|=km $ and $ |\ot(x)\cap\og(x)|=k $, for some $ k $ and $ m $, so that $ \ot(x) $ shortens in length by a factor of $ \frac{1}{k} $. Then by Lemma \ref{shortening}, 
\[
|\ot(x)\cap \og(x)|= \{x,g(x),g^2(x),\ldots, g^{k-1}(x)\}\,,
\]
where $ g(x)=T^m(x) $ and where $ k $ is minimal such that $ g^k(x)=x $. Then $ H=G_x=G_{g(x)}=gHg^{-1} $, so that $ g\in N_G(H) $. Further, $ g^k\in G_x=H $. Then $ k $ is the order of the coset $ gH $ in $ N_G(H)/H $. Hence, for any $ x\in X_H $ (and therefore $ x\in X_{[H]} $),
\begin{equation*}
|\ot(x)\cap\og(x)|\in \Delta\left(N_G(H)/H\right).
\end{equation*}

(ii) By the Orbit--Stabilizer--Theorem, we have
\begin{flalign*}
|\mathfrak{O}_G(x)| =[G:G_x] = [G: H].
\end{flalign*}
Then combining the above with Lemma \ref{glueing} and (i) the result follows.
\end{proof}

% splitting up orbits according to subgroups
% introducing notation for surviving, glueing, and shortening orbits

Now, let $ [H]\in\overline{P}(G) $, and define
\[ 
\Sigma_{[H]}=\left\{(\delta,\theta): \delta\in \Delta(N_G(H)/H), \theta=\tfrac{[G:H]}{\delta}\right\},
\]
Fix $ I=\{1\} $ to be the trivial subgroup of $ G $, so that
\[
\Sigma_{[I]}=\left\{(\delta,\theta): \delta\in \Delta(G), \theta=\tfrac{|G|}{\delta}\right\}.
\]
\noindent For $ \sigma=(\delta_\sigma,\theta_\sigma)\in\Sigma_{[H]} $, write $ \mathcal{O}^{[H]_\sigma}_n(T) $ for the set of orbits of length $ n $ in $ \mathbb{X}_{[H]} $ such that 
\[
|\ot(x)\cap\og(x)| = \delta_\sigma 
\qquad\text{and}\qquad \tfrac{[G:H]}{\delta_\sigma}=\theta_\sigma. 
\]
Then
\begin{itemize}
\item[(i)]
for $ \sigma\in\Sigma_{[H]} $ such that $ \delta_{\sigma}=1 $ and $ \theta_{\sigma}=[G:H]=1 $, $ \mathcal{O}^{[H]_\sigma}_n(T) $ is the set of surviving orbits of length $ n $ in $ \mathbb{X}_{[H]} $,
\item[(ii)]
for $ \sigma\in\Sigma_{[H]} $ such that $ \delta_{\sigma}=1 $ and $ \theta_{\sigma}= [G:H] >1 $, $ \mathcal{O}^{[H]_\sigma}_n(T) $ is the set of orbits of length $ n $ in $ \mathbb{X}_{[H]} $ which glue by a factor of $ \theta_{\sigma} $, and 
\item[(iii)]
for $ \sigma\in\Sigma_{[H]} $ such that $ \delta_{\sigma} >1 $, $ \mathcal{O}^{[H]_\sigma}_n(T) $ is the set of orbits of length $ n $ in $ \mathbb{X}_{[H]} $ which shorten in length by a factor of $ \tfrac{1}{\delta_\sigma} $ and glue by a factor of $ \theta_{\sigma} $. 
\end{itemize}
Hence, by \eqref{PartitioningPeriodicSet}, we obtain the disjoint union
\[
\mathcal{O}_n(T)= \bigsqcup_{[H]\in\overline{P}(G)} \,\bigsqcup_{\sigma\in\Sigma_{[H]}} 
  \mathcal{O}^{[H]_\sigma}_n(T).
  \]
It follows that
\begin{equation}
\label{TopOrbitDecomposition}
O_n(T)= \sum_{[H]\in\overline{P}(G)}\, \sum_{\sigma\in\Sigma_{[H]}} O^{[H]_\sigma}_n(T),   
\end{equation}
for all $ n\geq 1 $. Further, by the surjectivity of the map given in \eqref{orbitsurjectivity}, we have that 
\[
\mathcal{O}_n(T')= \bigsqcup_{[H]\in\overline{P}(G)}\, \bigsqcup_{\sigma\in\Sigma_{[H]}} 
   \pi\left( \mathcal{O}^{[H]_\sigma}_{\delta_\sigma n}(T)\right)
= \bigsqcup_{[H]\in\overline{P}(G)}\, \bigsqcup_{\sigma\in\Sigma_{[H]}} 
  \bigl( \tfrac{1}{\theta_\sigma}\bigr) \mathcal{O}^{[H]_\sigma}_{\delta_\sigma n}(T).
\]
It follows that 
\begin{equation}
\label{BottomOrbitDecomposition}
\end{equation}
\[
O_n(T')=  \sum_{[H]\in\overline{P}(G)}\, \sum_{\sigma\in\Sigma_{[H]}} 
\bigl(\tfrac{1}{\theta_\sigma}\bigr)O^{[H]_\sigma}_{\delta_\sigma n}(T), 
\]
for all $ n\geq 1 $. Here, we understand that we have the following constraints:
\begin{flalign}
\label{Constraint1}
&\theta_\sigma \mid O^{[H]_\sigma}_n(T);\\
\label{Constraint2}
&O^{[H]_\sigma}_n(T)=0 \text{ if } \delta_\sigma \nmid n.
\end{flalign}

% triangle example

\begin{example}
Let $ T $ be the doubling map defined on the two-torus $ X=\mathbb{R}^2/\mathbb{Z}^2 $. Let $ (\xr,\tr) $ be the quotient of $ (X,T) $ under the action of the dihedral group
\[
D_8 =\{1,\alpha,\alpha^2,\alpha^3,\tau,\alpha\tau,\alpha^2\tau,\alpha^3\tau \},
\]
where we define
\begin{flalign*}
\alpha(x,y)= (1-y,x)\qquad\text{and}\qquad\tau(x,y)=(y,x)
\end{flalign*}
for all $ (x,y)\in X $, so that $ \alpha $ acts on $ X $ through rotation by $ \frac{\pi}{2} $ and $ \tau $ acts through reflection. We have the following five conjugacy classes such that $ X_{[H]}\neq\emptyset $:
\begin{flalign*}
[G]&= \{D_8\}; \\
[\Omega]&= \{\{1,\alpha^2,\alpha\tau, \alpha^3\tau \}\}; \\
[\Phi]&= \{\{1,\alpha^3\tau\},\{1,\alpha\tau\}\}; \\
[\Pi]&= \{\{1,\tau\},\{1,\alpha^2\tau\}\}; \\
[I]&= \{\{1\}\},
\end{flalign*}
The partition pieces $ X_{[H]} $ of $ X $ and $ \xr_{[H]} $ of $ \xr $ are shown in Figure \ref{f3} and Figure \ref{f4}, respectively, where the coloured points give the indicated set.
\begin{figure}[h]
\begin{center}
\begin{tikzpicture}
% square 1
\draw (0,0)--(0,2)--(2,2)--(2,0)--(0,0);
% interior 1
\draw (1,0)--(1,2);
\draw (0,1)--(2,1);
\filldraw[blue] (0,0) circle [radius=0.075];
\filldraw[blue] (1,1) circle [radius=0.075];
% labels 1
\node[above] at (1,2.25) {$ X_{[G]} $};
\node[below] at (0,-0.1) {$ 0 $};
\node[below] at (1,-0.1) {$ \frac{1}{2} $};
\node[below] at (2,-0.1) {$ 1 $};
\node[left] at (0,1) {$ \frac{1}{2} $};
\node[left] at (0,2) {$ 1 $};
% square 2
\draw (2.5,0)--(2.5,2)--(4.5,2)--(4.5,0)--(2.5,0);
% interior 2
\draw (3.5,0)--(3.5,2);
\draw (2.5,1)--(4.5,1);
\filldraw[blue] (2.5,1) circle [radius=0.075];
\filldraw[blue] (3.5,0) circle [radius=0.075];
% labels 2
\node[above] at (3.5,2.25) {$ X_{[\Omega]} $};
\node[below] at (2.5,-0.1) {$ 0 $};
\node[below] at (3.5,-0.1) {$ \frac{1}{2} $};
\node[below] at (4.5,-0.1) {$ 1 $};
\node[left] at (2.5,1) {$ \frac{1}{2} $};
\node[left] at (2.5,2) {$ 1 $};
% square 3
\draw (5,2)--(7,2)--(7,0);
% interior 3
\draw[blue, very thick] (5,0)--(5,2);
\draw[blue, very thick] (6,0)--(6,2);
\draw[blue, very thick] (5,1)--(7,1);
\draw[blue, very thick] (5,0)--(7,0);
\filldraw[white] (5,0) circle [radius=0.075];
\draw[black] (5,0) circle [radius=0.075];
\filldraw[white] (6,0) circle [radius=0.075];
\draw[black] (6,0) circle [radius=0.075];
\filldraw[white] (7,0) circle [radius=0.075];
\draw[black] (7,0) circle [radius=0.075];
\filldraw[white] (7,1) circle [radius=0.075];
\draw[black] (7,1) circle [radius=0.075];
\filldraw[white] (6,1) circle [radius=0.075];
\draw[black] (6,1) circle [radius=0.075];
\filldraw[white] (6,2) circle [radius=0.075];
\draw[black] (6,2) circle [radius=0.075];
\filldraw[white] (5,1) circle [radius=0.075];
\draw[black] (5,1) circle [radius=0.075];
\filldraw[white] (5,2) circle [radius=0.075];
\draw[black] (5,2) circle [radius=0.075];
% labels 3
\node[above] at (6,2.25) {$ X_{[\Phi]} $};
\node[below] at (5,-0.1) {$ 0 $};
\node[below] at (6,-0.1) {$ \frac{1}{2} $};
\node[below] at (7,-0.1) {$ 1 $};
\node[left] at (5,1) {$ \frac{1}{2} $};
\node[left] at (5,2) {$ 1 $};
% square 4
\draw (7.5,0)--(7.5,2)--(9.5,2)--(9.5,0)--(7.5,0);
% interior 4
\draw[blue, very thick] (7.5,0)--(9.5,2);
\draw[blue, very thick] (7.5,2)--(9.5,0);
\draw[dotted] (7.5,1)--(9.5,1);
\draw[dotted] (8.5,0)--(8.5,2);
\filldraw[white] (7.5,0) circle [radius=0.075];
\draw[black] (7.5,0) circle [radius=0.075];
\filldraw[white] (9.5,0) circle [radius=0.075];
\draw[black] (9.5,0) circle [radius=0.075];
\filldraw[white] (7.5,2) circle [radius=0.075];
\draw[black] (7.5,2) circle [radius=0.075];
\filldraw[white] (9.5,2) circle [radius=0.075];
\draw[black] (9.5,2) circle [radius=0.075];
\filldraw[white] (8.5,1) circle [radius=0.075];
\draw[black] (8.5,1) circle [radius=0.075];
% labels 4
\node[above] at (8.5,2.25) {$ X_{[\Pi]} $};
\node[below] at (7.5,-0.1) {$ 0 $};
\node[below] at (8.5,-0.1) {$ \frac{1}{2} $};
\node[below] at (9.5,-0.1) {$ 1 $};
\node[left] at (7.5,1) {$ \frac{1}{2} $};
\node[left] at (7.5,2) {$ 1 $};
% square 5
\filldraw[blue] (10.06,0.05)-- (10.06,1.95) -- (11.95,1.95)--(11.95,0.06) -- (10.05,0.06);
\draw[black] (10,0)--(10,2)--(12,2)--(12,0)--(10,0);
% interior 5
\draw[white, thick] (10.05,0.05) -- (11.95,1.95);
\draw[white, thick] (10.05,1.95) -- (11.95,0.05);
\draw[white, thick] (10.05,1) -- (11.97,1);
\draw[white, thick] (11,0.03) -- (11,1.95);
% labels 5
\node[above] at (11,2.25) {$ X_{[I]} $};
\node[below] at (10,-0.1) {$ 0 $};
\node[below] at (11,-0.1) {$ \frac{1}{2} $};
\node[below] at (12,-0.1) {$ 1 $};
\node[left] at (10,1) {$ \frac{1}{2} $};
\node[left] at (10,2) {$ 1 $};
\end{tikzpicture}
\end{center}
\caption{}
\label{f3}
\end{figure}
\begin{figure}[h]
\begin{center}
\begin{tikzpicture}
% triangle 1
\draw (0,0)--(1.5,1.5)--(1.5,0)--(0,0);
% interior 1
\filldraw[blue] (0,0) circle [radius=0.075];
\filldraw[blue] (1.5,1.5) circle [radius=0.075];
% labels 1
\node[above] at (0.75,1.75) {$ \widehat{X}_{[G]} $};
\node[below] at (0,-0.1) {$ 0 $};
\node[below] at (1.5,-0.05) {$ \frac{1}{2} $};
\node[right] at (1.5,1.5) {$ \frac{1}{2} $};
% triangle 2
\draw (2.5,0)--(4,1.5)--(4,0)--(2.5,0);
% interior 2
\filldraw[blue] (4,0) circle [radius=0.075];
% labels 2
\node[above] at (3.25,1.75) {$ \widehat{X}_{[\Omega]} $};
\node[below] at (2.5,-0.1) {$ 0 $};
\node[below] at (4,-0.05) {$ \frac{1}{2} $};
\node[right] at (4,1.5) {$ \frac{1}{2} $};
% triangle 3
\draw (5,0)--(6.5,1.5);
% interior 3
\draw[blue, very thick] (6.5,1.5)--(6.5,0)--(5,0);
\filldraw[white] (5,0) circle [radius=0.075];
\draw[black] (5,0) circle [radius=0.075];
\filldraw[white] (6.5,0) circle [radius=0.075];
\draw[black] (6.5,0) circle [radius=0.075];
\filldraw[white] (6.5,1.5) circle [radius=0.075];
\draw[black] (6.5,1.5) circle [radius=0.075];
% labels 3
\node[above] at (5.75,1.75) {$ \widehat{X}_{[\Phi]} $};
\node[below] at (5,-0.1) {$ 0 $};
\node[below] at (6.5,-0.05) {$ \frac{1}{2} $};
\node[right] at (6.5,1.5) {$ \frac{1}{2} $};
% triangle 4
\draw (7.5,0)--(9,0)--(9,1.5);
% interior 4
\draw[blue, very thick] (7.5,0)--(9,1.5);
\filldraw[white] (7.5,0) circle [radius=0.075];
\draw[black] (7.5,0) circle [radius=0.075];
\filldraw[white] (9,1.5) circle [radius=0.075];
\draw[black] (9,1.5) circle [radius=0.075];
% labels 4
\node[above] at (8.25,1.75) {$ \widehat{X}_{[\Pi]} $};
\node[below] at (7.5,-0.1) {$ 0 $};
\node[below] at (9,-0.05) {$ \frac{1}{2} $};
\node[right] at (9,1.5) {$ \frac{1}{2} $};
% triangle 5
\filldraw[blue] (10.15,0.05)-- (11.425,1.35) -- (11.425,0.05)--(10.15,0.05);
\draw[black] (10,0)--(11.5,1.5)--(11.5,0)--(10,0);
% labels 5
\node[above] at (10.75,1.75) {$ \widehat{X}_{[I]} $};
\node[below] at (10,-0.1) {$ 0 $};
\node[below] at (11.5,-0.05) {$ \frac{1}{2} $};
\node[right] at (11.5,1.5) {$ \frac{1}{2} $};
\end{tikzpicture}
\end{center}
\caption{}
\label{f4}
\end{figure}
Noting that $ \widehat{\mathbb{X}}_{[\Omega]}=\emptyset $, we have the following four orbit behaviours:
\begin{flalign}
\label{TriangleMapOrbitbehaviour}
\begin{cases}
O^{[G]}_n(\T) = & O_n^{[G]_{(1,1)}}(T); \\
O_n^{[\Phi]}(\tr) = &\tfrac{1}{4}\, O^{[\Phi]_{(1,4)}}_n(T)+ \tfrac{1}{2}\,O^{[\Phi]_{(2,2)}}_{2n}(T);  \\
O_n^{[\Pi]}(\tr) = &\tfrac{1}{4}\, O^{[\Pi]_{(1,4)}}_n(T)+ \tfrac{1}{2}\,O^{[\Pi]_{(2,2)}}_{2n}(T);\\
O^{[I]}_n(\tr) = &\tfrac{1}{8}\, O^{[I]_{(1,8)}}_n(T)+ \tfrac{1}{4}\,O^{[I]_{(2,4)}}_{2n}(T)+ \tfrac{1}{2}\,O^{[I]_{(4,2)}}_{4n}(T).
\end{cases}
\end{flalign}
Then 
\begin{flalign*}
O_n(\tr)=\, O^{[G]}_n(\tr)+ O^{[\Phi]}_n(\tr) +  O^{[\Pi]}_n(\tr) + O^{[I]}_n(\tr),
\end{flalign*}
for all $ n\geq 1 $. Now, recall from Example \ref{ExampleCircleDoublingMap} that 
\begin{equation}
\label{TriangleMapAsymptoticGrowthRates}
F_n(T)\sim F_n(\tr)\quad\text{as $ n\to\infty $},
\end{equation}
showing the same asymptotic growth rates for orbits of $ (X,T) $ and its quotient system $ (\xr,\tr) $. By \eqref{TriangleMapOrbitbehaviour} and \eqref{TriangleMapAsymptoticGrowthRates}, we can deduce that the proportion of orbits which shorten in length is asymptotically zero. For a proof, the reader may refer to \cite{Zegowitz}.
\end{example}

% Upper and lower bounds

We have the following bounds for periodic points and orbits in quotient systems:

\begin{lemma}
\label{UpperAndLowerBounds}
Let $ (X',T') $ be the quotient system of $ (X,T) $ under the action of a finite group $ G $. Then, for any $ n\geq 1 $, 
\begin{itemize}
\item[($ \mathscr{B}_1 $)] $ \frac{F_n(T)}{|G|}\leq F_n(T')\leq F_n(T)+ \sum\limits_{\substack{\sigma\in\Sigma_{[I]}\\ \delta_\sigma >1}} \bigl( \frac{1}{\delta_\sigma \theta_\sigma}\bigr) F_{\delta_\sigma n}(T) $;   
\item[($ \mathscr{B}_2 $)] $ O_n(T')\leq O_n(T)+ \sum\limits_{\substack{\sigma\in\Sigma_{[I]}\\ \delta_\sigma >1}} \bigl( \frac{1}{\theta_\sigma}\bigr)O_{\delta_{\sigma}n}(T) $, and if $ n $ is such that $ \delta_{\sigma}\nmid n $, for any $ \sigma\in\Sigma_{[I]} $ such that $ \delta_\sigma >1 $, then $ O_n(T')\geq \frac{O_n(T)}{|G|} $.
\end{itemize}
\end{lemma}

\begin{proof}
($ \mathscr{B}_1 $) The lower bound for $ F_n(T') $ comes from the fact that the fibres of the topological factor map $ \pi $ have cardinality at most $ |G| $. Since $ \pi $ maps $ \mathcal{F}_n(T) $ to $ \mathcal{F}_n(T') $, we deduce that $ \frac{F_n(T)}{|G|}\leq F_n $. The bound is achieved whenever all orbits of length dividing $ n $ glue together by a factor of $ | G| $ and, for all $ \sigma\in\Sigma_{[I]} $ such that $ \delta_\sigma >1 $, no orbits of length $ \delta_\sigma n $ shorten in length by a factor of $ \frac{1}{\delta_\sigma} $. 

For the upper bound, note that if $ x $ is a point of period $ \delta_\sigma n $ for some $ \delta_\sigma >1 $, then $ \pi(x)\in\mathcal{F}_n(T') $ if and only if $ \ot(x) $ shortens in length by a factor of $ \frac{1}{\delta_\sigma} $ and glues by a factor of $ \theta_\sigma $, in which case $ x $ lies in a fibre of cardinality $ \delta_\sigma\theta_\sigma $. Then
\begin{flalign*}
F_n(T') &= \left|\mathcal{F}_n(T')\cap \pi\left(\mathcal{F}_n(T)\right)\right| +
\bigl|  \mathcal{F}_n(T')\cap \bigl[ \sum_{\substack{\sigma\in\Sigma_{[I]}\\ \delta_\sigma >1}} \pi( \mathcal{F}_{\delta_\sigma n}(T))\bigr]\bigr| \\
&\leq F_n(T) + \sum_{\substack{\sigma\in\Sigma_{[I]}\\ \delta_\sigma >1}} \bigl( \tfrac{1}{\delta_\sigma \theta_\sigma} \bigr) F_{\delta_\sigma n}(T).
\end{flalign*}

$ (\mathscr{B}_2) $ For the upper bound, we have
\begin{flalign*}
O_n(T') & = \sum_{[H]\in\overline{P}(G)}\sum_{\sigma\in\Sigma_{[H]}} \bigl(\tfrac{1}{\theta_\sigma}\bigr) O^{[H]_\sigma}_{\delta_\sigma n}(T) \\
&= \sum_{[H]\in\overline{P}(G)} \biggl[\,\,
\sum_{\substack{\sigma\in\Sigma_{[H]}\\ \delta_\sigma\theta_\sigma=1}} O^{[H]_\sigma}_n(T) + \sum_{\substack{\sigma\in\Sigma_{[H]}\\ \delta_\sigma=1,\theta_\sigma >1}} \bigl( \tfrac{1}{\theta_\sigma}\bigr) O^{[H]_\sigma}_n(T) + \sum_{\substack{\sigma\in\Sigma_{[H]}\\ \delta_\sigma>1}} \bigl(\tfrac{1}{\theta_\sigma}\bigr) O^{[H]_\sigma}_{\delta_\sigma n}(T)\biggr]\\
&\leq    \sum_{[H]\in\overline{P}(G)} \sum_{\sigma\in\Sigma_{[H]}}
\biggl[ O^{[H]_\sigma}_n(T) + \bigl( \tfrac{1}{\theta_\sigma}\bigr) O^{[H]_\sigma}_{\delta_\sigma n}(T)\biggr] \\
&\leq O_n(T) + \sum_{\substack{\sigma\in\Sigma_{[I]}\\ \delta_\sigma >1}} \bigl(\tfrac{1}{\theta_\sigma}\bigr) O_{\delta_\sigma n}(T),
\end{flalign*}
by \eqref{TopOrbitDecomposition} and \eqref{BottomOrbitDecomposition}. The bound is achieved whenever all orbits of length $ n $ survive and, for all $ \sigma\in\Sigma_{[I]} $ such that $ \delta_\sigma > 1 $, all orbits of length $ \delta_\sigma n $ shorten in length by a factor of $ \frac{1}{\delta_\sigma} $ and glue by a factor of $ \theta_\sigma $. 

Now, suppose that $ n $ is such that $ \delta_\sigma\nmid n $, for any $ \sigma\in\Sigma_{[I]} $ such that $ \delta_\sigma > 1 $. Then by \eqref{TopOrbitDecomposition} and the constraint given in \eqref{Constraint2}, we have that
\begin{equation}
\label{lowerboundorbits}
O_n(T) =  \sum_{[H]\in\overline{P}(G)} \biggl[\,\, \sum_{\substack{\sigma\in\Sigma_{[H]}\\ \delta_\sigma\theta_\sigma=1}} O^{[H]_\sigma}_n(T)+ \sum_{\substack{\sigma\in\Sigma_{[H]}\\ \delta_\sigma=1,\theta_\sigma>1}} O^{[H]_\sigma}_n(T) \biggr]
\end{equation}
Then by \eqref{BottomOrbitDecomposition}, \eqref{lowerboundorbits}, and the fact that $ \theta\leq |G| $, we have that the lower bound for $ (\mathscr{B}_2) $ is given by
\begin{flalign*}
O_n(T') &=  \sum_{[H]\in\overline{P}(G)} \sum_{\sigma\in\Sigma_{[H]}} \bigl(\tfrac{1}{\theta_\sigma}\bigr) O^{[H]_\sigma}_{\delta_\sigma n}(T) \\
&= \sum_{[H]\in\overline{P}(G)}  \biggl[\,\, \sum_{\substack{\sigma\in\Sigma_{[H]}\\ \delta_\sigma\theta_\sigma=1}} O^{[H]_\sigma}_n(T)
+ \sum_{\substack{\sigma\in\Sigma_{[H]}\\ \delta_\sigma=1,\theta_\sigma >1}} \bigl(\tfrac{1}{\theta_\sigma}\bigr) O^{[H]_\sigma}_n(T)
+ \sum_{\substack{\sigma\in\Sigma_{[H]}\\ \delta_\sigma >1}} \bigl(\tfrac{1}{\theta_\sigma}\bigr) O^{[H]_\sigma}_{\delta_\sigma n}(T)\biggr] \\
&\geq \sum_{[H]\in\overline{P}(G)} \biggl[\,\, \sum_{\substack{\sigma\in\Sigma_{[H]}\\ \delta_\sigma\theta_\sigma=1}} O^{[H]_\sigma}_n(T)
+ \sum_{\substack{\sigma\in\Sigma_{[H]}\\ \delta_\sigma=1, \theta_\sigma >1}} \bigl(\tfrac{1}{|G|}\bigr) O^{[H]_\sigma}_n(T) \biggr] \\
&\geq    \frac{1}{|G|} \sum_{[H]\in\overline{P}(G)} \biggl[\,\, \sum_{\substack{\sigma\in\Sigma_{[H]}\\ \delta_\sigma\theta_\sigma=1}} O^{[H]_\sigma}_n(T) + \sum_{\substack{\sigma\in\Sigma_{[H]}\\ \delta_\sigma=1,\theta_\sigma >1}} O^{[H]_\sigma}_n(T) \biggr] \\
&= \frac{O_n(T)}{|G|}.
\end{flalign*}
Again, the bound is achieved whenever all orbits of length $ n $ glue together by a factor of $ |G| $ and, for $ \sigma\in\Sigma_{[I]} $ such that $ \delta_\sigma >1 $, no orbits of length $ \delta_\sigma n $ shorten in length by a factor of $ \frac{1}{\delta_\sigma} $. 
\end{proof}

\begin{remark}
Note that the upper bound for $ (\mathscr{B}_1) $ is not optimal but sufficient for the purpose of this paper. 
\end{remark}

\begin{remark}
Lemma 8 does not give a lower bound for $ (\mathscr{B}_2) $ since, in that case, we have only the trivial lower bound $ O_n(T')\geq 0 $. This lower bound is achieved whenever all orbits of length $ n $ shorten in length and, for $ \sigma\in\Sigma_{[I]} $ such that $ \delta_\sigma >1 $, no orbits of length $ \delta_\sigma n $ shorten in length by a factor of $ \frac{1}{\delta_\sigma} $. 
\end{remark}

Now, let
\[
\nabla=\max\{\delta:\delta\in\Delta(G)\}
\]
be the largest order of any element in $ G $. Denote by $ H_\nabla $ a maximal subgroup of $ G $ such that $ \nabla\in \Delta(N_G(H_\nabla)/H_\nabla) $, and set
\[
\Theta= \frac{[G:H_\nabla]}{\nabla}
\]
to be the minimum value for $ \theta_\sigma $ such that $ \sigma= (\nabla,\theta_\sigma) $. Then Lemma \ref{UpperAndLowerBounds} gives an immediate result concerning the logarithmic growth rate in $ F_n(T') $:

\begin{corollary}
\label{CorollaryLimSup}
Let $ (X',T') $ be the quotient system of the topological dynamical system $ (X,T) $ under the action of a finite group $ G $, and suppose there exists $ \eta >0 $ such that $ \limsup\limits_{n\to\infty} \frac{1}{n}\log F_n(T)=\eta $. Then 
\[
\eta \leq \limsup\limits_{n\to\infty} \frac{1}{n}\log F_n(T') \leq \nabla \eta.
\]
\end{corollary}
\begin{proof}
Let $ \epsilon >0 $. Then there exists $ N>0 $ such that when $ n>N $, we have that
\begin{equation}
\label{limsupdefinition}
\eta-\epsilon < \tfrac{1}{n}\log  F_n(T) < \eta+\epsilon.                                       
\end{equation}
By Lemma \ref{UpperAndLowerBounds} and \eqref{limsupdefinition}, we obtain
\begin{flalign*}
\tfrac{1}{n} \log F_n(T')   &\geq   \tfrac{1}{n} \log\left( \tfrac{F_n(T)}{|G|}\right) 
> \eta+ \epsilon - \tfrac{1}{n} \log |G|.
\end{flalign*}
It follows that
\begin{equation}
\label{limsuplowerbound}
\limsup_{n\to\infty} \tfrac{1}{n}\log F_n(T') \geq \eta.                                          
\end{equation}

Further, for the upper bound, we have 
\begin{flalign*}
\tfrac{1}{n}\log F_n(T') &\leq  \tfrac{1}{n} \log\biggl( F_n(T)+ \sum_{\substack{\sigma\in\Sigma_{[I]}\\ \delta_\sigma >1}} \bigl(\tfrac{1}{\delta_\sigma \theta_\sigma}\bigr) F_{\delta_\sigma n}(T) \biggr) 
\leq \tfrac{1}{n} \log\left(\tfrac{ |\Sigma_{[I]}|}{\nabla\Theta} \right) 
+ \tfrac{1}{n} \log F_{\nabla n}(T) \\
&< \tfrac{1}{n} \log\left(\tfrac{|\Sigma_{[I]}|}{\nabla\Theta}\right) + \nabla (\eta+\epsilon),
\end{flalign*}
by Lemma \ref{UpperAndLowerBounds} and \eqref{limsupdefinition}. It follows that
\begin{equation}
\label{limsupupperbound}
\limsup_{n\to\infty} \tfrac{1}{n}\log F_n(T') \leq \nabla \eta.                              
\end{equation}
Combining \eqref{limsuplowerbound} and \eqref{limsupupperbound}, the result follows.  
\end{proof}
Later, in Corollary \ref{CorollaryOrbitGrowthRates}, we will show that any growth rate in between the bounds of Corollary \ref{CorollaryLimSup} can be achieved.

\section{The Super Basic Lemma}

We now observe that if we are free to choose a pair of topological dynamical systems, then \eqref{Constraint1} and \eqref{Constraint2} are the only constraints on the orbit behaviour under a finite group action. We have the following result:
\begin{proposition}
\label{SuperBasicLemma}
Let $ G $ be a finite group, and let $ \bigl(b_n^{[H]_\sigma}\bigr)_{n=1}^\infty $ be a sequence of non-negative integers, for $ [H]\in\overline{P}(G) $ and $ \sigma\in\Sigma_{[H]} $, such that $ b^{[G]_{(1,1)}}_1 \geq 1 $. Define $ \bigl(a_n^{[H]_\sigma}\bigr)_{n=1}^\infty $ by
\[
a_n^{[H]_\sigma}=
\begin{cases}
\theta_\sigma \,b^{[H]_\sigma}_{n/ \delta_\sigma}, & \text{if } \delta_\sigma\mid n, \\
0, & \text{otherwise}.
\end{cases}
\]
Further, define
\begin{flalign*}
a_n&= \sum_{[H]\in\overline{P}(G)}\,\sum_{\sigma\in\Sigma_{[H]}} a_n^{[H]_\sigma},
\\
b_n&= \sum_{[H]\in\overline{P}(G)}\,\sum_{\sigma\in\Sigma_{[H]}} b_n^{[H]_\sigma},
\end{flalign*}
for all $ n\geq 1 $. Then there exist a topological dynamical system $ (X,T) $ and an action of $ G $ on $ X $ which commutes with $ T $ such that
\[
O_n(T)=a_n 
\qquad\text{and}\qquad
O_n(T')=b_n,
\]
for all $ n\geq 1 $.
\end{proposition}

Note that setting $ G= C_2 $ recovers the basic lemma of \cite{Halving}.

\begin{proof}
Let $ H $ be a representative for $ [H]\in \overline{P}(G) $, and suppose $ h_\sigma\in N_G(H) $ is such that $ h_\sigma H $ has order $ \delta_\sigma $ in $ N_G(H)/ H $. Then letting $ S_{[H]} $ be a set of representatives for the coset space $ G/ N_G(H) $, we have that
\[
|S_{[H]}|= \theta_\sigma. 
\]

We define
\[
X= \bigsqcup_{n\geq 1} X_n
\]
where $ X_n $ is the union of closed orbits of length $ n $, and set 
\[
X^{[H]_\sigma}_n= S_{[H]}\times \bigl\{1,2, \ldots, b^{[H]_\sigma}_{n/ \delta_\sigma} \bigr\}\times \mathbb{Z}/ n\mathbb{Z},
\]
and
\[
X_n=  \bigsqcup_{[H]\in\overline{P}(G)}\, \bigsqcup_{\sigma\in\Sigma_{[H]}} X_n^{[H]_\sigma},
\]
where we understand that $ \bigl\{1,2,\ldots, b^{[H]_\sigma}_{n/ \delta_\sigma}\bigr\}=\emptyset $ if $ b^{[H]_\sigma}_{n/ \delta_\sigma}=0 $. 

We note that for $ s\in S_{[H]} $, we can write $ gs= s' h^l_\sigma h' $, where $ s'\in S_{[H]} $ is uniquely determined, $ l\in\N_0 $ is taken modulo $ \delta_\sigma $, and $ h'\in H $. Then, for $ g\in G $ and $ x=(s,i,m)\in X $, we can define an action of $ G $ on $ X $ as follows:
\[
g(x)=\bigl( s', i, m+ \tfrac{l n}{\delta_\sigma}\bigr),
\]
for all $ x\in X $. Further, we define $ T:X\to X $ by
\[
T(x)=(s, i, m+1\,(\text{mod } n)),
\]
for all $ x\in X $. We check that $ g $ indeed satisfies all criteria to be an action of $ G $ which commutes with $ T $. Hence, by construction, it follows that
\[
O_n(T)= a_n
\qquad\text{and}\qquad
O_n(T')= b_n,
\]
for all $ n\geq 1 $, where $ (X',T') $ is the quotient of $ (X, T) $ under the action of $ G $.

It remains to show that $ X $ can be given the structure of a metric space with respect to which $ T $ is a homeomorphism. First, we take a point in $ X^{[G]_{(1,1)}}_1 $ (which is non-empty by construction) and call this point $ \infty $. We then define a metric for $ x\in X_n $ with $ x\neq \infty $ and $ x'\in X_{n'} $ by
\[
d(x,x')=
\begin{cases}
0, & \text{if } x=x', \\
\frac{1}{n}, & \text{if } x'=\infty, \\
\frac{1}{\min\{n,n'\}}, & \text{otherwise} .
\end{cases}
\]
Then given any open set $ U $ such that $ \infty\in U $, there exists $ N\geq 1 $ such that $ \bigsqcup\limits_{n\geq N} X_n \subseteq U $. Then $ X\backslash U= \bigsqcup\limits_{n<N} X_n $ is finite. It follows that $ X $ is compact. Moreover, since $ T $ preserves each set $ X_n $ and the point $ \infty $, we have that $ T $ is a homeomorphism. 
\end{proof}

\section{Orbit Growth Rates in Quotient Systems}
\label{SectionOrbitGrowthRates}

Proposition \ref{SuperBasicLemma} shows that any two sequences $ (a_n) $ and $ (b_n) $ of non-negative integers which satisfy the combinatorial constraints \eqref{Constraint1} and \eqref{Constraint2} arise as the orbit count of a pair of topological dynamical systems related to each other by a finite group action. However, it is possible to impose conditions directly on the sequences $ (a_n) $ and $ (b_n) $ to guarantee that the combinatorial constraints are satisfied. The following result gives some sufficient conditions:
\begin{theorem}
\label{SequencesProposition}
Let $ G $ be a finite group and let $ (a_n) $ be a sequence of non-negative integers with $ a_1\geq 1 $ such that there exists $ N>0 $ with $ \frac{a_{\nabla n}}{\Theta}\geq a_n $ for $ n\geq N $. Further, let $ (b_n) $ be any sequence of non-negative integers such that $ b_1> \frac{a_1}{|G|} $ and
\[
\begin{cases}
\frac{a_n}{|G|}\leq b_n \leq a_n, & \text{for $ n<N $}, \\
a_n\leq b_n \leq \frac{a_{\nabla n}}{\Theta}, & \text{for $ n\geq N $}.
\end{cases}
\]
Then there exist a topological dynamical system $ (X,T) $ and an action of $ G $ on $ X $ which commutes with $ T $ such that
\[
O_n(T)=a_n\qquad\text{and}\qquad O_n(T')=b_n,
\]
for all $ n\geq 1 $.
\end{theorem}

\begin{proof}
We recursively construct sequences $ \bigl(b_n^{[G]_{(1,1)}}\bigr)_{n= 1}^\infty $, $ \bigl(b_n^{[I]_{(1,|G|)}}\bigr)_{n= 1}^\infty $, and $ \bigl(b_n^{[H_\nabla]_{(\nabla,\Theta)}}\bigr)_{n= 1}^\infty $ of non-negative integers such that 
\begin{flalign*}
b_n &= b^{[G]_{(1,1)}}_n + b^{[I]_{(1,|G|)}}_n + b^{[H_\nabla]_{(\nabla, \Theta)}}_n, \\
a_n &=  b^{[G]_{(1,1)}}_n + |G|\,b^{[I]_{(1,|G|)}}_n + \Theta\, b^{[H_\nabla]_{(\nabla, \Theta)}}_{n/ \nabla},
\end{flalign*}
for all $ n\geq 1 $, where we understand that $ b^{[H_\nabla]_{(\nabla, \Theta)}}_{n/ \nabla}=0 $ if $\nabla\nmid n $. Now, let $ n<N $, and choose
\[
\begin{cases}
b^{[H_\nabla]_{(\nabla, \Theta)}}_n =0, \\
b^{[I]_{(1, |G|)}}_n = \left\lceil \tfrac{a_n-b_n}{|G|-1}\right\rceil, \\
b^{[G]_{(1,1)}}_n = b_n- b^{[I]_{(1,|G|)}}_n,
\end{cases}
\]
and $ b^{[G]_{(1,1)}}_1 \geq 1 $ (since $ b_1 > \frac{a_1}{|G|} $). Note that the above are non-negative integers by definition. 

Next, let $ n\geq N $, and choose
\[
\begin{cases}
b^{[I]_{(1, |G|)}}_n=0, \\
b^{[G]_{(1,1)}}_n= a_n- \Theta\, b^{[H_\nabla]_{(\nabla, \Theta)}}_{n/ \nabla}, \\
b^{[H_\nabla]_{(\nabla, \Theta)}}_n= b_n- b^{[G]_{(1,1)}}_n.
\end{cases}
\]
Again, the above are non-negative integers by definition. Applying Proposition \ref{SuperBasicLemma}, the result follows. 
\end{proof}

As a consequence of Theorem \ref{SequencesProposition}, we have the following result concerning asymptotic growth rates for orbits which shows that any growth rate in between the bounds of Corollary \ref{CorollaryLimSup} can be achieved.

\begin{corollary}
\label{CorollaryOrbitGrowthRates}
Suppose $ 1<\lambda $, $ \eta $, and $ c>0 $ are such that either
\[
\begin{cases}
\eta= \lambda\text{ and } c\geq \frac{1}{|G|}, \text{ or} \\
\lambda < \eta < \lambda^{\nabla}, \text{ or} \\
\eta= \lambda^{\nabla} \text{ and }c\leq \frac{1}{\Theta}.
\end{cases}
\]
Then there exist a topological dynamical system $ (X,T) $ and an action of $ G $ on $ X $ which commutes with $ T $ such that
\begin{flalign*}
&\,\, O_n(T) \sim \lambda^n \quad\text{as } n\to\infty,\\
&O_n(T')\sim c\eta^n\quad\text{as } n\to\infty.
\end{flalign*}
\end{corollary}
Note that setting $ G=C_2 $ recovers Corollary 10 in \cite{Halving}.

\begin{proof}
Let $ N $ be such that $ c\eta^N < \frac{\lambda^{\nabla N}}{\Theta} $ and define sequences $ (b_n) $ and $ (a_n) $ such that 
\[
a_n= \left\lceil \lambda^n \right\rceil
\quad\text{and}\quad
b_n= 
\begin{cases}
\left\lceil c\lambda^n \right\rceil, & \text{if $ n<N $}, \\
\left\lceil c\eta^n \right\rceil, & \text{if $ n\geq N $},
\end{cases}
\]
for $ n\geq 1 $. Hence, we satisfy the hypothesis of Theorem \ref{SequencesProposition}. The result follows. 
\end{proof}

\section{Concluding Remarks and Questions}

\begin{enumerate}
\item In \cite{WardPuri}, it was observed that for any sequence $ (a_n) $ of non-negative integers there exists a topological dynamical system $ (X,T) $ such that $ O_n(T)= a_n $, for all $ n\geq 1 $. This result was extended by Windsor in \cite{Windsor}, showing $ (X,T) $ to have a smooth model, that is, there exists $ (X,T) $ with $ T $ a $ C^\infty $ diffeomorphism of the two-torus such that $ O_n(T)= a_n $, for all $ n\geq 1 $. Then given sequences $ (a_n) $ and $ (b_n) $ as in the hypothesis of Proposition \ref{SuperBasicLemma}, is there a smooth model $ (X,T) $ and an action of a finite group $ G $ on $ X $ which commutes with $ T $ such that $ O_n(T)= a_n $ and $ O_n(T')= b_n $, for all $ n\geq 1 $?

\item
Example \ref{ExampleCircleDoublingMap} and Example \ref{ExampleTriangleMap} both exhibit the same asymptotic growth rates for orbits of the dynamical system and its quotient, and therefore, they are not representative for the general case of Proposition \ref{SuperBasicLemma} which gives a wide but restricted range of growth rates which can be achieved in the quotient system. We observe that the proof for Proposition \ref{SuperBasicLemma} uses a more abstract combinatorial construction of $ (X,T) $ very unlike the natural constructions of Example \ref{ExampleCircleDoublingMap} and Example \ref{ExampleTriangleMap}. Both examples occur in a natural setting: They have easy to define spaces and maps and the groups chosen are natural (non-trivial) choices for the given space and map. Do examples occurring in a natural setting, for example, classes of systems like group automorphisms, subshifts of finite type, expanding maps on an interval, always exhibit the same growth rates for orbits in the quotient system?
\end{enumerate}

\section*{Acknowledgements}
This work is an extension of the author's PhD thesis supervised by Prof. Shaun Stevens and Prof. Tom Ward, and it was supported by the Swedish Research Council (VR Grant 2010/5905). The author would like to thank Prof. Shaun Stevens and Prof. Tom Ward for their valuable feedback concerning this work. 

\bibliographystyle{plain}
\bibliography{ref}

\end{document}